\documentclass[11pt,letterpaper,twoside,reqno,nosumlimits]{amsart}

\allowdisplaybreaks

\usepackage[usenames,dvipsnames]{xcolor}
\usepackage{fancyhdr}
\usepackage{amsmath,amsfonts,amsbsy,amsgen,amscd,amssymb,amsthm}
\usepackage{subfig}
\usepackage{url}

\usepackage{mathtools}

\mathtoolsset{showonlyrefs}

\usepackage[font=small,margin=0.25in,labelfont={sc},labelsep={colon}]{caption}

\usepackage{tikz}
\usepackage{microtype}
\usepackage{enumitem}

\definecolor{dark-gray}{gray}{0.3}
\definecolor{dkgray}{rgb}{.4,.4,.4}
\definecolor{dkblue}{rgb}{0,0,.5}
\definecolor{medblue}{rgb}{0,0,.75}
\definecolor{rust}{rgb}{0.5,0.1,0.1}

\usepackage[colorlinks=true]{hyperref}

\hypersetup{urlcolor=rust}
\hypersetup{citecolor=dkblue}
\hypersetup{linkcolor=dkblue}

\usepackage{setspace}

\usepackage{graphicx}
\usepackage{booktabs,longtable,tabu} 
\usepackage{multirow} 

\usepackage{float}

\usepackage[full]{textcomp}

\usepackage[scaled=.98,sups]{XCharter}
\usepackage[scaled=1.04,varqu,varl]{inconsolata}
\usepackage[type1]{cabin}
\usepackage[charter,vvarbb,scaled=1.07]{newtxmath}
\usepackage[cal=boondoxo]{mathalfa}
\usepackage[T1]{fontenc}

\usepackage{bm} 

\graphicspath{{figures/}}

\newtheorem{theorem}{Theorem}[section]
\newtheorem{lemma}[theorem]{Lemma}

\newtheorem{proposition}[theorem]{Proposition}

\newtheorem{corollary}[theorem]{Corollary}

\theoremstyle{definition}

\newtheorem{example}[theorem]{Example}
\newtheorem{remark}[theorem]{Remark}

\numberwithin{equation}{section} 
\numberwithin{figure}{section}
\numberwithin{table}{section}

\floatstyle{plaintop}
\newfloat{recipe}{thp}{lor}
\floatname{recipe}{Recipe}
\numberwithin{recipe}{section}

\providecommand{\mathbold}[1]{\bm{#1}}  
                                
\renewcommand{\phi}{\varphi}

\newcommand{\suml}{\sum\limits}

\newcommand{\econst}{\mathrm{e}}
\newcommand{\iunit}{\mathrm{i}}

\newcommand{\Id}{\mathbf{I}}

\providecommand{\mathbbm}{\mathbb} 

\newcommand{\R}{\mathbbm{R}}

\newcommand{\N}{\mathbbm{N}}
\newcommand{\Z}{\mathbbm{Z}}

\newcommand{\diff}[1]{\mathrm{d}{#1}}
\newcommand{\idiff}[1]{\, \diff{#1}}

\newcommand{\Prob}[1]{\mathbbm{P}\left\{{#1}\right\}}

\newcommand{\Expect}{\operatorname{\mathbbm{E}}}

\DeclareMathOperator{\Var}{Var}

\newcommand{\mtx}[1]{\mathbold{#1}}

\newcommand{\triplenorm}[1]{{\left\vert\kern-0.25ex\left\vert\kern-0.25ex\left\vert #1
    \right\vert\kern-0.25ex\right\vert\kern-0.25ex\right\vert}}

\newcommand{\mL}{\mathcal{L}}

\newcommand{\E}{\Expect}

\usepackage[numbers]{natbib}

\begin{document}

\title[Bernstein Inequalities for Markov Processes]{Bernstein-type Inequalities for Markov Chains\\ and Markov Processes: A Simple and Robust Proof}

\author[D. Huang, X. Li]{De Huang$^1$, Xiangyuan Li$^2$}
\thanks{$^1$School of Mathematical Sciences, Peking University. E-mail: dhuang@math.pku.edu.cn}
\thanks{$^2$School of Mathematical Sciences, Peking University. E-mail: lixiangyuan23@stu.pku.edu.cn}

\begin{abstract}
We establish a new Bernstein-type deviation inequality for general (non-reversible) discrete-time Markov chains via an elementary approach. More robust than existing works in the literature, our result only requires the Markov chain to satisfy an iterated Poincar\'e inequality. Moreover, our method can be readily generalized to continuous-time Markov processes.
\end{abstract}

\maketitle

\section{Introduction}
Concentration inequalities describe the probability that the average of a sequence of random variables is close to its expected value. Motivated by the profound impact of these inequalities on many scientific applications, significant efforts have been made to extend concentration results from independent random variables to dependent ones. In particular, the concentration of Markov chains has been an extensively studied topic. Let $\{Z_k\}_{k\in\mathbb{N}_+}$ be a discrete-time Markov chain on a state space $\Omega$ with a unique invariant distribution $\mu$. Let $L_{2,\mu}$ be the Hilbert space of all real-valued functions on $\Omega$ endowed with the inner product $\langle f,g \rangle_{\mu}=\mathbb{E}_{\mu}[f\cdot g]$,
and let $L_{2,\mu}^0=\{f \in L_{2,\mu}: \E_\mu[f]=0\}$ be the mean-zero subspace of $L_{2,\mu}$. Our goal of this paper is to bound the deviation probability $\Prob{|\frac{1}{n}\sum_{k=1}^n f(Z_k)| \geq \delta}$ for any bounded function $f\in L_{2,\mu}^0$.

In general, the convergence speed of the empirical mean $\frac{1}{n}\sum_{k=1}^n f(Z_k)$ depends on certain spectral properties of the corresponding transition operator $\mtx{P}$. In particular, ergodicity (in different senses) for non-reversible Markov chains has been proved to be guaranteed and quantified by various notions of spectral gaps of $\mtx{P}$ \cite{fill1991eigenvalue,montenegro2006mathematical,kontoyiannis2012geometric,paulin2015concentration,chatterjee2025spectral}. In this paper, we use the smallest nonzero singular value of $\mtx{\Id}-\mtx{P}$ (with respect to $L_{2,\mu}$) to quantify the ergodicity of $\{Z_k\}_{k\in\mathbb{N}_+}$ and establish a Bernstein-type inequality for Markov chains with a nonzero iterated Poincar\'e gap. To be specific, we define the \textit{iterated Poincar\'e gap} (IP gap) of $\mtx{P}$ as
\begin{equation}
    \eta_p:=\inf\limits_{h \in L_{2,\mu}^0,\ h \neq \mtx{0}}\frac{\|(\mtx{\Id}-\mtx{P})h\|_{\mu}}{\|h\|_{\mu}}.
\end{equation}
Here and below, $\|h\|_{\mu}=\langle h,h \rangle_{\mu}^{1/2}$. Then for any $\delta>0$ and any function $f$ such that $\E_{\mu}[f]=0$, $|f|\leq M$, and $\Var_\mu[f]\leq \sigma^2$, we can obtain the following Bernstein-type tail bound:
\begin{equation}\label{eqt:mainresult_simple}
\Prob{\left|\frac{1}{n}\sum\limits_{k=1}^n f(Z_k) \right| \geq \delta } \lesssim \exp\left(-\frac{n\,\eta_p\,\delta^2}{4M\sqrt{(2+6\eta_p)^2\sigma^2+\delta^2}}\right).
\end{equation}
We call $\eta_p$ the iterated Poincar\'e gap since it is related to the iterated Poincar\'e inequality:
\begin{equation}
    \Var_\mu[h] \leq \eta_p^{-2 }\E_\mu\left[\left(\left(\mtx{\Id}-\mtx{P}\right) h\right)^2\right],\quad  \text{for any}\,\ h\in L_{2,\mu}.
\end{equation}
The definition of this IP gap coincides with that of the non-reversible spectral gap first formally introduced by Chatterjee in \cite{chatterjee2025spectral}, where he used this quantity to establish ergodicity for non-reversible Markov chains, generalizing similar results for reversible Markov chains based on the usual spectral gap. For instance, Chatterjee showed that the empirical variance along a non-reversible Markov chain of length $n$ can be bounded by $1/(n\eta_p)$, in the same spirit as in the reversible case. The reader may thus also refer to the IP gap as the \textit{Chatterjee gap} or the \textit{non-reversible spectral gap} and can find a more comprehensive introduction to the topic in \cite{chatterjee2025spectral}. Nevertheless, we want to remark that the proof ideas and the results of this paper are novel and were developed independently and unaware of the work \cite{chatterjee2025spectral} by Chatterjee\footnote{The core techniques and the main results of this paper were already established by the first author DH in 2020 and were introduced to Joel Tropp (Caltech) in a private communication in the same year. We then spent a few years trying to improve the estimates and to extend the results to non-commutative (matrix) settings (though failed) before we finally released this paper in 2024.}, and we use the name ``iterated Poincar\'e gap'' to emphasize the critical but simple use of the iterated Poincar\'e inequality in our proof. 

There have been numerous studies that establish concentration inequalities for Markov chains in terms of certain spectral gaps of the transition operators \cite{lezaud1998chernoff,leon2004optimal,chung2012chernoff,miasojedow2014hoeffding,paulin2015concentration,rao2019hoeffding,jiang2018bernstein,fan2021hoeffding}.
When $\{Z_k\}_{k\in\mathbb{N}_+}$ is reversible, the corresponding $\mtx{P}$ is self-adjoint under the $L_{2,\mu}$ inner product. In this case, the spectral gap of $\mtx{P}$ is given by 
\begin{equation}\label{eqt:definition_of_spectral_gap}
\eta:=1-\sup _{h \in L_{2,\mu}^0,\ h \neq \mtx{0}} \frac{\langle h,\mtx{P}h \rangle_{\mu}}{\|h\|_{\mu}^2}.
\end{equation}
In other words, $\eta$ is the gap between $1$ and the second largest eigenvalue of $\mtx{P}$ (when the spectrum of $\mtx{P}$ is well-defined). It is well-known that, in the reversible case, a nonzero spectral gap of $\mtx{P}$ implies ergodicity and concentration of the Markov chain.

For non-reversible Markov chains, various notions of spectral gaps have been proposed as substitutes for the usual spectral gap in the study of concentration inequalities. A commonly used one is the \textit{absolute $L_2$-spectral gap} (or simply the \textit{absolute spectral gap}), which is defined as
\begin{equation}\label{eqt:absolute_operator_gap}
\eta_a:= 1-\sup _{h \in L_{2,\mu}^0,\ h \neq \mtx{0}}  \frac{\|\mtx{P} h\|_{\mu}}{\|h\|_{\mu}}.
\end{equation}
If $\mtx{P}$ is considered as a linear operator on $L_{2,\mu}$, then $\eta_a$ is the gap between $1$ and the operator norm of $\mtx{P}$ over the subspace $L_{2,\mu}^0$ (the latter is also the second largest singular value of $\mtx{P}$ under the $L_{2,\mu}$ inner product). Alternatively, $\eta_a$ can be understood as the usual spectral gap of the self-adjoint operator $(\mtx{P}^*\mtx{P})^{1/2}$, where $\mtx{P}^*$ is the adjoint of $\mtx{P}$ with respect to the $L_{2,\mu}$ inner product. Moreover, in some literature (e.g. \cite{lezaud1998chernoff}) an alternative definition, 
\begin{equation}
\eta_m:= 1-\sup _{h \in L_{2,\mu}^0,\ h \neq \mtx{0}}  \frac{\|\mtx{P} h\|_{\mu}^2}{\|h\|_{\mu}^2},
\end{equation}
is employed, which is the usual spectral gap of $\mtx{P}^*\mtx{P}$, the multiplicative reversiblization of $\mtx{P}$ (see also \cite{fill1991eigenvalue,montenegro2006mathematical,chatterjee2025spectral}). It is easily seen that these two gaps $\eta_a$ and $\eta_m$ are always comparable since $\eta_m = 1-(1-\eta_a)^2 = (2-\eta_a)\eta_a$ by definition, so it does not matter which definition is used. To the best of our knowledge, most of the existing concentration inequalities for non-reversible Markov chains require a nonzero $\eta_a$ (or $\eta_m$) \cite{lezaud1998chernoff,miasojedow2014hoeffding,rao2019hoeffding,jiang2018bernstein,fan2021hoeffding}. Unfortunately, though a nonzero $\eta_a$ implies ergodicity of the Markov chain, the reverse is not true in general: a nonzero absolute spectral gap cannot be guaranteed even for irreducible Markov chains (Simple examples can be found in Appendix \ref{sec:simple_example}). 

Note that in some literature \cite{paulin2015concentration,levin2017markov,chatterjee2025spectral} the name ``absolute spectral gap'' refers to a different quantity defined as 
\begin{equation}
\gamma_a := 1-|\tilde{\lambda}_2(\mtx{P})|,
\end{equation}
where $\tilde\lambda_2(\mtx{P})$ denotes the second largest eigenvalue of $\mtx{P}$ in absolute value. This definition of $\gamma_a$ coincides with that of $\eta_a$ in \eqref{eqt:absolute_operator_gap} for reversible Markov chains. However, it is not so favorable to establish concentration inequalities in terms of $\gamma_a$ for non-reversible Markov chains. In principle, one can first bound the deviation probability (the left-hand side of \eqref{eqt:mainresult_simple}) by $\exp(-\delta^2/O(t_{\mathrm{mix}}) )$ (see e.g. \cite{paulin2015concentration}), where $t_{\mathrm{mix}}$ is the so called mixing time of the Markov chain, and then bound $t_{\mathrm{mix}}$ from above by a multiple of $1/\gamma_a$ \cite{paulin2015concentration,jerison2013general}. Nevertheless, it is proved in \cite{jerison2013general} that $t_{\mathrm{mix}}\lesssim |\Omega|/\gamma_a$ (where $|\Omega|$ is the size of the state space $\Omega$) for general Markov chains with an example showing the sharpness of this bound. Such undesirable dependence on the size of the state space makes $\gamma_a$ a less useful quantity in the non-reversible case. Therefore, we will not discuss $\gamma_a$ in this paper, and we will always refer to $\eta_a$ defined in \eqref{eqt:absolute_operator_gap} as the absolute spectral gap.

A generalization of the absolute spectral gap, called the \textit{pseudo spectral gap}, was first introduced in \cite{paulin2015concentration} to establish concentration inequalities for non-reversible Markov chains. It is defined as 
\begin{equation}
\eta_{ps} := \sup_{k\in \N_+}\frac{1}{k}\left(1- \sup _{h \in L_{2,\mu}^0,\ h \neq \mtx{0}}  \frac{\|\mtx{P}^k h\|_{\mu}^2}{\|h\|_{\mu}^2}\right) = \sup_{k\in \N_+}\frac{\eta\big((\mtx{P}^*)^k\mtx{P}^k\big)}{k},
\end{equation}
where $\eta((\mtx{P}^*)^k\mtx{P}^k)$ is the usual spectral gap of the self-adjoint operator $(\mtx{P}^*)^k\mtx{P}^k$. Though it could be difficult to compute this quantity exactly, one can obtain an easy lower bound by taking the supremum over $k\in\{1,\dots,K\}$ for some finite $K$. Nevertheless, this relatively more complicated definition of the pseudo spectral gap makes it less user-friendly in practice. Advanced theoretical studies, estimations, and applications of the pseudo spectral gap can be found in \cite{wolfer2019estimating,wolfer2024improved}. 

Another way to define the spectral gap for a non-reversible Markov chain is by simply using the usual spectral gap of the self-adjoint operator $(\mtx{P}+\mtx{P}^*)/2$, the additive symmetrization of $\mtx{P}$ (see \cite{fill1991eigenvalue,montenegro2006mathematical,chatterjee2025spectral}). To be specific, the \textit{symmetric spectral gap} of $\mtx{P}$ is defined as 
\begin{equation}
\eta_s:=1-\sup _{h \in L_{2,\mu}^0,\ h \neq \mtx{0}} \frac{\langle h,\mtx{P}h \rangle_{\mu}}{\|h\|_{\mu}^2}=1-\sup _{h \in L_{2,\mu}^0,\ h \neq \mtx{0}} \frac{\left\langle h,\frac{1}{2}\left(\mtx{P}+\mtx{P}^*\right)h \right\rangle_{\mu}}{\|h\|_{\mu}^2},
\end{equation}
which can be seen as a natural generalization of \eqref{eqt:definition_of_spectral_gap}. This quantity is employed to obtain concentration inequalities in \cite{fan2021hoeffding,jiang2018bernstein}. However, there seems to be a technical gap in their proof of a critical lemma involving the symmetric spectral gap (which will be explained and discussed in Appendix \ref{sec:gap}). Other than these unjustified results, the symmetric spectral gap has not been used in concentration inequalities for non-reversible Markov chains.

Note that in the reversible case, it always holds that $\eta=\eta_p=\eta_s$, and if the second largest eigenvalue of $\mtx{P}$ in absolute value is positive, then it also holds that $\eta_p = \eta_a$. The reason we prefer the IP gap is that the relations $\eta_p\geq \eta_s\geq \eta_a$ (Lemma \ref{lem:gap_comparison}) and $\eta_p\geq \eta_{ps}/2$ (see \cite{chatterjee2025spectral}) always hold for any non-reversible Markov chain, meaning that the IP gap is the most robust one among all these commonly used gaps for quantifying ergodicity. In fact, for any finite-state irreducible Markov chain, the corresponding IP gap $\eta_p$ is always nonzero. Moreover, there are simple examples where $\eta_p>0$ but the ratios $\eta_a/\eta_p$ and $\eta_{ps}/\eta_p$ can be arbitrarily small or even zero (see Appendix \ref{sec:simple_example}). Consequently, our results are more robust, user-friendly, and applicable to a much wider range of Markov chains. 

We establish \eqref{eqt:mainresult_simple} by recursively bounding the moment-generating function of $\sum_{k=1}^n f(Z_k)$ in a way that only an iterated Poincar\'e inequality is needed, providing a more straightforward and elementary method compared to existing approaches. Furthermore, our method can be naturally generalized to continuous-time Markov processes. To be specific, let $(Z_t)_{t\geq0}$ be a continuous-time Markov process on $\Omega$ with an invariant distribution $\mu$, and let $(\mtx{P}_t)_{t\geq 0}$ be the associated Markov semigroup. In this case, we can similarly define the IP gap $\eta_p$ for $(\mtx{P}_t)_{t\geq 0}$ via an iterated Poincar\'e inequality and establish the following Bernstein-type inequality for any function $f$ that satisfies the same assumptions as above:
\begin{equation}
\Prob{\left|\frac{1}{t}\int_0^t f(Z_s)\idiff s\right| \geq \delta } \lesssim \exp\left(-\frac{t\,\eta_p\,\delta^2}{4M\sqrt{4\sigma^2+\delta^2}}\right), \quad \text{for any}\,\ \delta \geq 0.
\end{equation}
We remark that though the proof in the continuous-time case is parallel to that in the discrete-time case, the former is cleaner and easier to read for some technical reasons. We hence suggest reading the proof in the continuous-time case first to quickly get to the key idea of our approach.\\

The rest of this paper is organized as follows. In Section \ref{sec:diecrete}, we prove the Bernstein-type inequality \eqref{eqt:mainresult_simple} for discrete-time Markov chains. Section \ref{sec:continuous} extends our result to continuous-time Markov processes. Appendix \ref{sec:simple_example} discusses the relations between the IP gap and the other spectral gaps mentioned above. Finally, a potential gap in the proofs in \cite{fan2021hoeffding,jiang2018bernstein} is discussed in Appendix \ref{sec:gap}.

\section{Concentration inequalities for discrete-time Markov chains}\label{sec:diecrete}
In this section, we establish our Bernstein-type inequality \eqref{eqt:mainresult_simple} for discrete-time Markov chains.

\subsection{Setting}
Throughout this paper, $\Omega$ is a Polish space equipped with a probability measure $\mu$. $L_{2,\mu}$ denotes the Hilbert space of all real-valued functions on $\Omega$ endowed with the inner product 
\begin{equation}
    \langle f,g \rangle_{\mu}=\mathbb{E}_{\mu}[\,f\cdot g\,],
\end{equation}
and $L_{2,\mu}^0=\{f \in L_{2,\mu}: \E_\mu[f]=0\}$ denotes the mean-zero subspace of $L_{2,\mu}$. We define $\|f\|_\mu := \langle f,f \rangle_{\mu}^{1/2}$. Correspondingly, we have 
\[\Var_\mu[f] = \E_\mu[(f-\E_\mu f)^2] = \|f-\E_\mu f\|_\mu^2.\]

Let $\{Z_k\}_{k\in \mathbb{N}_+}$ be a Markov chain (not necessarily reversible) on $\Omega$ whose invariant distribution is $\mu$, and let $\mtx{P}$ be the associated transition operator given by
\[\mtx{P}h(z) = \E\left[h(Z_2)\,|\,Z_1 = z\right],\quad \text{for any} \,\ z\in \Omega.\]
By the definition of the invariant distribution, we have 
\begin{equation}
    \E_{\mu}[\mtx{P}h]= \E_{\mu}[h], \quad \text{for any $h\in L_{2,\mu}$}.
\end{equation}
Let $\mtx{L}$ be the Laplacian operator of $\mtx{P}$ given by
\begin{equation}
\mtx{L}h(z) := (\mtx{P}-\mtx{\Id})h(z), \quad \text{for all $z\in \Omega$},
\end{equation}
where $\Id$ is the identity operator. We define the \textit{iterated Poincar\'e gap (IP gap)} of $\mtx{P}$ as
\begin{equation}
    \eta_p:=\inf\limits_{h \in L_{2,\mu}^0, h \neq \mtx{0}}\frac{\|\mtx{L}h\|_{\mu}}{\|h\|_{\mu}}.
\end{equation}
When $\{Z_k\}_{k\in \mathbb{N}_+}$ is a Markov chain on a finite state space, $\mtx{L}$ is a finite-dimensional matrix, and the corresponding IP gap $\eta_p$ is the second smallest singular value of $\mtx{L}$ under the $L_{2,\mu}$ inner product. 

In what follows, we will assume $\mtx{P}$ admits an IP gap $\eta_p>0$. That is, there exists a constant $\eta_p>0$ such that an iterated Poincar\'e inequality holds on $L_{2,\mu}$:
\begin{equation}\label{eqt:2ndPoincare_discrete}
\Var_\mu[h] \leq \eta_p^{-2 }\E_\mu\left[\left(\mtx{L} h\right)^2\right],  \quad\text{for any $h\in L_{2,\mu}$}.
\end{equation}

\subsection{Exponential moment bound} Let $f: \Omega\to \mathbb{R}$ be a function such that $\E_\mu [f] = 0$, $|f|\leq M$, and $\Var_\mu[f]\leq \sigma^2$. The goal of this subsection is to bound the moment generating function (mgf) of $\sum_{k=1}^{n}f(Z_k)$, which is defined as
\begin{equation}
m(\theta):= \E_{Z_1\sim\mu}\left[\econst^{\theta \sum_{k=1}^{n}f(Z_k) }\right].
\end{equation}
To simplify notations, we will first consider the case $\theta=1$. After that, we can bound $m(\theta)$ for arbitrary $\theta>0$ by replacing the function $f$ with $\theta f$. 
For $n\geq 1$, define
\[F\left(Z_{n_1:n_2}\right) := \sum\limits_{k=n_1}^{n_2}f(Z_k) \quad \text{and}\quad G_n(z) := \E\left[\econst^{ F(Z_{1:n})}\,|\, Z_1 = z\right],\quad z\in \Omega.\]
Our goal is to bound the quantity
\begin{equation}
a_n := \E_{Z_1\sim\mu}\left[\econst^{F(Z_{1:n})}\right]= \E_{z\sim\mu}[G_n(z)].
\end{equation}

To start, we compute the action of $\mtx{L}$ on the function $G_n$ as follows.

\begin{lemma}\label{lem:L_of_Gn} It holds that 
\[\mtx{L} G_n(z) = \int_0^1\E\left[(f(Z_{n+1})-f(Z_1))\econst^{sF(Z_{2:(n+1)}) + (1-s)F(Z_{1:n})}\,|\,Z_1 = z\right]\idiff s.\]
\end{lemma}

\begin{proof} By the definition of $\mtx{L}$, we can compute that 
\begin{align*}
\mtx{L} G_n(z) &= \E\left[G_n(Z_2)\,|\,Z_1 = z\right] - G_n(z) \\
&= \E\left[\econst^{F(Z_{2:(n+1)})} - \econst^{F(Z_{1:n})}\,\big|\,Z_1 = z\right] \\
&= \E\left[(f(Z_{n+1})-f(Z_1))\int_0^1\econst^{sF(Z_{2:(n+1)}) + (1-s)F(Z_{1:n})}\idiff s\,\Big|\,Z_1 = z\right],
\end{align*}
which is the claimed result.
\end{proof}

With the above lemma, we can bound $a_n$ as follows given that $f$ satisfies some smallness condition.

\begin{proposition}\label{prop:bound_of_a_n}
Let $f: \Omega\to \mathbb{R}$ be a function such that $\E_\mu [f] = 0$, $|f|\leq M$, and $\Var_\mu[f]\leq \sigma^2$. Suppose that  $1-4M^2/\eta_p^2>0$, then for any $n\in \mathbb{N}_+$,
\begin{equation}
a_n:= \E_{Z_1\sim\mu}\left[\econst^{F(Z_{1:n})}\right]\leq \exp\left(n\sigma M\cdot\frac{2+6\eta_p}{c\,\eta_p} \right),
\end{equation}
where $c:=(1-4M^2/\eta_p^2)^{1/2}$.
\end{proposition}

\begin{proof}
By the definition of $\eta_p$, it is apparent that $\eta_p\leq 2$. Owing to the assumption $1-4M^2/\eta_p^2>0$, one has $M<\eta_p/2\leq 1$. We shall bound $a_n$ by recursion. For any $n\geq 1$, we have
\begin{equation}\label{eqt:discrete_step_1}
\begin{split}
a_{n+1} - a_n
&= \E_{Z_1\sim\mu}\left[\econst^{F\left(Z_{1:(n+1)}\right)}\right] - \E_{Z_1\sim\mu}\left[\econst^{F(Z_{1:n})}\right]\\
&= \E_{Z_1\sim\mu}\left[\econst^{F\left(Z_{1:(n+1)}\right)}\right] - \E_{Z_1\sim\mu}\left[\econst^{F(Z_{2:(n+1)})}\right]\\
&= \int_0^1\E_{Z_1\sim\mu}\left[f(Z_1)\econst^{F\left(Z_{1:n}\right)+f(Z_{n+1})-(1-s)f(Z_1)}\right]\idiff s \\
&= \E_{Z_1\sim\mu}\left[f(Z_1)\econst^{F(Z_{1:n})}\right] + \int_0^1\E_{Z_1\sim\mu}\left[f(Z_1)(\econst^{f(Z_{n+1})-(1-s)f(Z_1)}-1)\econst^{F(Z_{1:n})}\right]\idiff s\\
&\leq \E_{Z_1\sim\mu}\left[f(Z_1)\econst^{F(Z_{1:n})}\right] +  \int_0^1\left(\econst^{M(2-s)}-1\right)\idiff s\cdot\E_{Z_1\sim\mu}\left[|f(Z_1)|\econst^{F(Z_{1:n})}\right]\\
&\leq \E_{z\sim\mu}\left[f(z)G_n(z)\right] + 4M\cdot \E_{z\sim\mu}\left[|f(z)|G_n(z)\right],
\end{split}
\end{equation}
where the last inequality is due to the numerical fact that 
\begin{equation}
\frac{\econst^{2M}-\econst^{M}}{M}-1\leq 4M, \quad \text{for any $M\in(0,1]$}. 
\end{equation}
Since $\E_\mu [f] = 0$, we have
\begin{equation}\label{eqt:discrete_step_2}
\E_{\mu}[fG_n] = \E_{\mu}\left[f\left(G_n-\E_{\mu}[G_n]\right)\right] \leq \Var_\mu[f]^{1/2}\Var_\mu[G_n]^{1/2} \leq \frac{\sigma}{\eta_p} \left(\E_\mu\big[(\mtx{L} G_n)^2\big]\right)^{1/2},
\end{equation}
where the last inequality follows from \eqref{eqt:2ndPoincare_discrete} and the assumption $\Var_\mu[f]\leq \sigma^2$. Similarly, we deduce
\begin{equation}\label{eqt:discrete_step_3}
\E_{\mu}[|f|G_n] = \E_{\mu}\left[|f|\left(G_n-\E_{\mu}[G_n]\right)\right]+ \E_{\mu}\left[|f|\E_{\mu}[G_n]\right]\leq \frac{\sigma}{\eta_p} \left(\E_\mu\big[(\mtx{L} G_n)^2\big]\right)^{1/2}+\sigma\cdot a_n.
\end{equation}
By Lemma \ref{lem:L_of_Gn}, we can compute that 
\begin{align*}
|\mtx{L} G_n(z)| &\leq \int_0^1\E\left[|f(Z_{n+1})-f(Z_1)|\econst^{sF(Z_{2:(n+1)}) + (1-s)F(Z_{1:n})}\,|\,Z_1 = z\right]\idiff s \\
&\leq 2M \int_0^1\E\left[s\econst^{F(Z_{2:(n+1)})} + (1-s)\econst^{F(Z_{1:n})}\,|\,Z_1 = z\right]\idiff s \\
&= M(\mtx{P}G_n(z) + G_n(z)).
\end{align*}
It follows that
\[\E_\mu\left[(\mtx{L} G_n)^2\right]\leq M^2 \E_\mu\left[(\mtx{P}G_n + G_n)^2\right]\leq 2M^2\E_\mu\left[(\mtx{P}G_n)^2\right] +  2M^2\E_\mu\big[G_n^2\big]\leq 4M^2\E_\mu\big[G_n^2\big].\]
To continue, we use the above result to compute that 
\[\E_\mu\big[G_n^2\big]= \Var_\mu[G_n] + \left(\E_\mu[G_n]\right)^2 \leq  \frac{1}{\eta_p^2}\E_\mu\left[(\mtx{L} G_n)^2\right] + a_n^2 \leq  \frac{4M^2}{\eta_p^2}\E_\mu\big[G_n^2\big] + a_n^2.\]
Using the assumption $c^2= 1 - 4M^2/\eta_p^2>0$, we obtain 
\[\E_\mu[G_n^2]\leq \frac{a_n^2}{c^2}.\] 
Finally, we have
\begin{equation}
\left(\E_\mu\big[(\mtx{L} G_n)^2\big]\right)^{1/2}\leq 2M\left(\E_\mu\big[G_n^2\big]\right)^{1/2}\leq \frac{2M}{c}\cdot a_n.
\end{equation}
We then combine \eqref{eqt:discrete_step_1}, \eqref{eqt:discrete_step_2}, \eqref{eqt:discrete_step_3}, and the above to obtain
\begin{align*}
a_{n+1} &\leq a_n + \frac{2\sigma M}{c\eta_p}\cdot a_n + 4\sigma M\left(\frac{2M}{c\eta_p}+1\right)\cdot a_n \\
&\leq \left(1+\frac{2\sigma M}{c\eta_p}+\frac{4\sigma M\big(\sqrt{1-c^2}+c\big)}{c}\right)\cdot a_n\\
&\leq a_n\cdot \exp\left( \sigma M\cdot\frac{2+6\eta_p}{c\,\eta_p} \right).
\end{align*}
We have used the fact that $\sqrt{1-c^2}+c\leq \sqrt{2}<3/2$ for any $c\in[0,1]$. Unrolling this recursion yields $a_n\leq \econst^{(n-1)\sigma M(2+6\eta_p)/(c\eta_p)}a_1$. In fact, using the same argument we can show that $a_1\leq \econst^{\sigma M(2+6\eta_p)/(c\eta_p)}$. Finally, we obtain 
\[a_n\leq \exp\left( n\sigma M\cdot\frac{2+6\eta_p}{c\,\eta_p} \right)\] 
for all $n\geq 1$.
\end{proof}

As a direct corollary, we can bound the mgf of $\sum_{k=1}^{n}f(Z_k)$ for any $0<\theta<\eta_p/(2M)$. With $f$ replaced by $\theta f$, the smallness assumption on $M$ in Proposition \ref{prop:bound_of_a_n} now becomes a smallness condition on $\theta$.

\begin{corollary}\label{cor:discrete_mgf_bound}
Let $f: \Omega\to \mathbb{R}$ be a function such that $\E_\mu [f] = 0$, $|f|\leq M$, and $\Var_\mu[f]\leq \sigma^2$. For any $\theta\in \R$ with $|\theta|<\eta_p/(2M)$, it holds that 
\begin{equation}
\E_{Z_1\sim\mu}\left[\econst^{\theta F(Z_{1:n})}\right]\leq \exp\left(n\sigma M\theta^2 \cdot\frac{2+6\eta_p}{c(\theta)\,\eta_p} \right),
\end{equation}
where $c(\theta):=(1-4\theta^2M^2/\eta_p^2)^{1/2}$.
\end{corollary}

\begin{proof}
Let $g=\theta f$, then $g$ satisfies $\E_\mu [g] = 0$, $|g|\leq \theta M$, and $\Var_\mu[g]\leq \theta^2\sigma^2$. One then repeats the proof of Proposition \ref{prop:bound_of_a_n} with $g$ in place of $f$ to get the claimed result.
\end{proof}

\subsection{Bernstein-type tail bound}
Using the preceding bound on the mgf of $\sum_{k=1}^n f(Z_k)$, we can now establish our Bernstein-type inequality for discrete-time Markov chains as follows.

\begin{theorem}\label{thm:Bernstein_discrete}
Let $\left\{Z_k\right\}_{k\in \mathbb{N}_+}$ be a Markov chain on $\Omega$ with an invariant distribution $\mu$. Suppose that the corresponding transition operator $\mtx{P}$ admits an IP gap $\eta_p>0$. Let $\nu$ be the initial distribution of $\left\{Z_k\right\}_{k\in \mathbb{N}_+}$, i.e. $Z_1\sim\nu$, such that $\nu$ is absolutely continuous with respect to $\mu$ and
that $\nu/\mu$ has a finite $p$-moment for some $p\in[1,+\infty]$. Let $f\in \Omega\to \mathbb{R}$ be a function such that $\E_\mu [f] = 0$, $|f|\leq M$, and $\Var_\mu[f]\leq \sigma^2$ for some $M,\sigma\in (0,+\infty)$. Then, for all $n\in\N_+$ and all $\delta\geq 0$,
\begin{equation}
\Prob{\left|\frac{1}{n}\sum\limits_{k=1}^nf(Z_k) \right| \geq \delta } \leq 2\|\nu/\mu\|_{L_{p,\mu}}\exp\left(-\frac{n\,\eta_p\,\delta^2}{4qM\sqrt{\left(2+6\eta_p\right)^2\sigma^2+\delta^2}}\right),
\end{equation}
where $q=p/(p-1)\in[1,+\infty]$.
\end{theorem}

\begin{proof} For any $\theta>0$, we have
\begin{align*}
\Prob{\frac{1}{n}\sum\limits_{k=1}^nf(Z_k) \geq \delta } &\leq \econst^{-\theta n\delta} \E_{Z_1\sim\nu}\left[\econst^{\theta F(Z_{1:n})}\right]\\
&=\econst^{-\theta n\delta} \E_{Z_1\sim\mu}\left[\left(\nu(Z_1)/\mu(Z_1)\right)\econst^{\theta F(Z_{1:n})}\right]\\
&\leq \econst^{-\theta n\delta} \|\nu/\mu\|_{L_{p,\mu}}\left(\E_{Z_1\sim\mu}\left[\econst^{q\theta F(Z_{1:n})}\right]\right)^{1/q}.
\end{align*}
We can use Corollary \ref{cor:discrete_mgf_bound} to deduce
\[\E_{Z_1\sim\mu}\left[\econst^{q\theta F(Z_{1:n})}\right]\leq  \exp\left(n\sigma Mq^2\theta^2 \cdot\frac{2+6\eta_p}{c(q\theta)\,\eta_p} \right),\]
given that $c(q\theta)^2=1 - 4q^2\theta^2M^2/\eta_p^2> 0$. In particular, we may choose
\begin{equation}
\theta = \frac{\delta\eta_p}{2qM\sqrt{\left(2+6\eta_p\right)^2\sigma^2+\delta^2}}\,,
\end{equation}
which gives 
\[c(q\theta) = \frac{(2+6\eta_p)\sigma}{\sqrt{\left(2+6\eta_p\right)^2\sigma^2+\delta^2}}.\]
With this choice of $\theta$, we obtain
\begin{align*}
\Prob{\frac{1}{n}\sum\limits_{k=1}^nf(Z_k) \geq \delta }&\leq  \|\nu/\mu\|_{L_{p,\mu}}\exp\left(- n\delta\theta+  \frac{n\sigma M q\theta^2\left(2+6\eta_p\right)}{c(q\theta)\, \eta_p }\right)\\ 
&=\|\nu/\mu\|_{L_{p,\mu}} \exp\left(-\frac{n\eta_p\delta^2}{4qM\sqrt{\left(2+6\eta_p\right)^2\sigma^2+\delta^2}}\right).
\end{align*}
We can derive the same bound for $\Prob{\frac{1}{n}\sum_{k=1}^nf(Z_k) \leq -\delta}$ in a similar way. Combining these estimates yields the desired bound on the whole deviation probability with an extra factor of $2$.
\end{proof}

Recall that the classical Bernstein-type inequalities for general Markov chains, e.g. those obtained in \cite{lezaud1998chernoff,jiang2018bernstein}, take the typical form
\begin{equation}\label{eqt:traditional_result}
    \Prob{\left|\frac{1}{n}\sum\limits_{k=1}^n f(Z_k) \right| \geq \delta } \lesssim \exp\left(-\frac{C\,n\eta_a\delta^2}{\sigma^2+M\delta}\right),
\end{equation}
where $C$ is some universal constant, and $\eta_a$ is the absolute spectral gap. Paulin \cite{paulin2015concentration} obtained an improvement of \eqref{eqt:traditional_result} in the sense that $\eta_a$ is replaced by the pseudo spectral $\eta_{ps}$. In comparison, our Bernstein-type inequality takes the form
\begin{equation}\label{eqt:our_result}
    \Prob{\left|\frac{1}{n}\sum\limits_{k=1}^n f(Z_k) \right| \geq \delta } \lesssim \exp\left(-\frac{C\,n\eta_p\delta^2}{M(\sigma+\delta)}\right).
\end{equation}
In practice, the deviation $\delta$ of the empirical mean is typically of order $1/\sqrt{n}$ and thus is much smaller than $\sigma$ when $n$ is large. Besides, when $\sigma\ll M$, Bernstein-type inequalities (\eqref{eqt:traditional_result} and \eqref{eqt:our_result}) are more preferable than Hoeffding-type inequalities (as those in \cite{miasojedow2014hoeffding,rao2019hoeffding,fan2021hoeffding} that also use $\eta_a$). For parameters in this range, i.e. $\delta\ll \sigma \ll M$, the traditional result \eqref{eqt:traditional_result} (with a variance proxy $\sigma^2$) could be sharper than our result \eqref{eqt:our_result} (with a variance proxy $M\sigma$). However, the significance of our new result \eqref{eqt:our_result} lies in that it only relies on an IP gap $\eta_p$ of $\mtx{P}$, while \eqref{eqt:traditional_result} relies on an absolute spectral $\eta_a$ or a pseudo spectral gap $\eta_{ps}$. In Appendix \ref{sec:simple_example}, we show that $\eta_p\geq \eta_a$ and $\eta_p\geq \eta_{ps}/2$ for any (non-reversible) Markov transition operator $\mtx{P}$. Furthermore, when $\mtx{P}$ is a finite-state irreducible transition matrix, $\eta_p$ is always nonzero, while $\eta_a$ or $\eta_{ps}$ can be arbitrarily small compared to $\eta_p$ or even zero (see the examples in Appendix \ref{sec:simple_example}). In this sense, our result \eqref{eqt:our_result} is more robust than the classical one \eqref{eqt:traditional_result}. 

Nevertheless, it would still be great to improve our result by establishing a Bernstein inequality in the standard form \eqref{eqt:traditional_result} but also with the IP gap $\eta_p$ in use. Unfortunately, though this difficulty seems to be merely technical, our current method does not seem to be able to overcome it. We hope to dig into this problem in future works.

\section{Concentration inequalities for continuous-time Markov processes}\label{sec:continuous}
In this section, we generalize our Bernstein-type inequality to continuous-time Markov processes.

\subsection{Setting}
Let $\Omega$, $\mu$, and $L_{2,\mu}$ be defined as in the previous section. Let $(Z_t)_{t\geq0}$ be a Markov process (not necessarily reversible) on $\Omega$ whose stationary distribution is $\mu$, and let $(\mtx{P}_t)_{t\geq 0}$ be the associated Markov semigroup given by
\[\mtx{P}_th(z) = \E\left[h(Z_t)\,|\,Z_0=z\right],\quad \text{for all $z\in \Omega$}.\]
 Let $\mL$ be the infinitesimal generator of $(\mtx{P}_t)_{t\geq 0}$:
\[\mL h(z) := \lim_{t\rightarrow 0+0} \frac{\mtx{P}_th(z)-h(z)}{t},\quad \text{for all $z\in \Omega$}.\]
We say $(\mtx{P}_t)_{t\geq 0}$ admits an IP gap $\eta_p>0$, if the following iterated Poincar\'e inequality holds on $L_{2,\mu}$:
\begin{equation}\label{eqt:2ndPoincare_continuous}
\Var_\mu[h] \leq \eta_p^{-2 }\E_\mu\left[\left(\mL h\right)^2\right],\quad \text{for any $h\in L_{2,\mu}$}.
\end{equation}
In the following, we will always assume $(\mtx{P}_t)_{t\geq 0}$ admits an IP gap $\eta_p>0$.

\subsection{Exponential moment bound}
Again, let $f\in \Omega\to \mathbb{R}$ be such that $\E_\mu [f] = 0$, $|f|\leq M$, and $\Var_\mu[f]\leq \sigma^2$. To avoid non-essential discussions, it is also reasonable to assume that $f(Z_t)$ is weakly right-continuous in $t$: for any $z\in \Omega$, any $s\geq 0$, and any bounded random variable $X = X(\{Z_t\}_{t\geq 0})$,
\begin{equation}\label{eqt:weak_continuous}
\lim_{\tau\rightarrow 0+0} \E\left[f(Z_{s+\tau})X\big|\, Z_0=z\right] = \E\left[f(Z_s)X\big|\, Z_0=z\right].
\end{equation}
Note that this condition only requires mild regularity of $f$ and $\{Z_t\}_{t\geq 0}$ and is satisfied in most of the practical applications.

The goal of this subsection is to bound the mgf of $\int_0^tf(Z_s)\idiff s$, which is defined as
\begin{equation}
m(\theta):=\E_{Z_0\sim\mu}\left[\econst^{\theta\int_0^tf(Z_s)\idiff s}\right].
\end{equation}
Similar to the discrete-time case, we first consider the case $\theta=1$.
For $t\geq 0$, define
\[F\left(Z_{[a,b]}\right) := \int_a^bf(Z_s)\idiff s\quad \text{and}\quad G_t(z) := \E\left[\econst^{ F\left(Z_{[0,t]}\right)}\,\big|\, Z_0 = z\right],\quad z\in \Omega.\]
Our goal is to bound the quantity
\[a(t) := \E_{Z_0\sim\mu}\left[\econst^{F\left(Z_{[0,t]}\right)}\right]= \E_{z\sim\mu}[G_t(z)],\quad t\geq0.\]
We can establish a continuous version of Lemma \ref{lem:L_of_Gn} as follows.

\begin{lemma}\label{lem:L_of_G} It holds that 
\[\mL G_t(z) = \E\left[(f(Z_t)-f(Z_0))\econst^{F\left(Z_{[0,t]}\right)}\,|\, Z_0 = z\right].\]
\end{lemma}

\begin{proof} By the definition of $\mL$, we can compute that 
\begin{align*}
\mL G_t(z) &= \lim_{\tau\rightarrow 0+0}\frac{\mtx{P}_\tau G_t(z) - G_t(z)}{\tau}\\
&= \lim_{\tau\rightarrow 0+0} \E\left[\frac{G_t(Z_\tau)-G_t(Z_0)}{\tau}\,\Big|\,Z_0 = z\right] \\
&=  \lim_{\tau\rightarrow 0+0} \E\left[\frac{\econst^{F(Z_{[\tau,t+\tau]})}-\econst^{F(Z_{[0,t]})}}{\tau}\,\Big|\,Z_0 = z\right] \\
&= \E\left[(f(Z_t)-f(Z_0))\econst^{F\left(Z_{[0,t]}\right)}\,\big|\, Z_0 = z\right],
\end{align*}
which is the claimed result. We have used \eqref{eqt:weak_continuous} for calculating the limit above.
\end{proof}

Next, we can bound $a(t)$ using Gr\"onwall's inequality. To be specific, we have the following estimate.

\begin{proposition}\label{prop:bound_of_a_t}
Let $f:\Omega\to \mathbb{R}$ be a function such that $\E_\mu [f] = 0$, $|f|\leq M$, and $\Var_\mu[f]\leq \sigma^2$. Suppose that  $1-4M^2/\eta_p^2>0$, then for any $t\geq 0$,
\begin{equation}
a(t):= \E_{Z_0\sim\mu}\left[\econst^{F\left(Z_{[0,t]}\right)}\right]\leq \exp\left(\frac{2\sigma M\,t}{c\,\eta_p}\right),
\end{equation}
where $c:=(1-4M^2/\eta_p^2)^{1/2}$.
\end{proposition}

\begin{proof}
Note that $a(0)=1$. Taking the derivative of $a(t)$ (by using \eqref{eqt:weak_continuous}) for $t\geq 0$ gives 
\begin{equation}\label{eqt:continuous_step_1}
\begin{split}
a'(t) &= \E_{Z_0\sim\mu}\left[f(Z_t)\econst^{F\left(Z_{[0,t]}\right)}\right] \\
&= \E_{z\sim\mu}\left[\E\left[f(Z_t)\econst^{F\left(Z_{[0,t]}\right)}\,\big|\, Z_0 = z\right]\right]\\
&= \E_{z\sim\mu}\left[\E\left[f(Z_0)\econst^{F\left(Z_{[0,t]}\right)}\,\big|\, Z_0 = z\right]\right] + \E_{z\sim\mu}[\mL G_t(z)] \\
&= \E_{z\sim\mu}[f(z)G_t(z)].
\end{split}
\end{equation}
The third equality above follows from Lemma \ref{lem:L_of_G}, and the fourth equality follows from that $\E_\mu[\mL g]=0$ for any function $g$. Since $\E_\mu [f] = 0$, we have 
\begin{equation}\label{eqt:continuous_step_2}
\E_{\mu}\left[fG_t\right] = \E_{\mu}\left[f(G_t-\E_{\mu}[G_t])\right]\leq \Var_\mu[f]^{1/2}\Var_\mu[G_t]^{1/2} \leq \frac{\sigma}{\eta_p} \left(\E_\mu\big[(\mL G_t)^2\big]\right)^{1/2},
\end{equation}
where the last inequality follows from \eqref{eqt:2ndPoincare_continuous}. Using Lemma \ref{lem:L_of_G}, we have
\begin{align*}
\E_\mu\left[(\mL G_t)^2\right] &= \E_{z\sim\mu}\left[\left(\E\left[(f(Z_t)-f(Z_0))\econst^{F\left(Z_{[0,t]}\right)}\,|\, Z_0 = z\right]\right)^2\right]\\
&\leq 4M^2\E_{z\sim\mu}\left[\left(\E\left[\econst^{F\left(Z_{[0,t]}\right)}\,|\, Z_0 = z\right]\right)^2\right] = 4M^2\E_\mu\big[G_t^2\big].
\end{align*}
To continue, we compute that 
\[\E_\mu\big[G_t^2\big]= \Var_\mu[G_t] + \E_\mu[G_t]^2 \leq  \frac{1}{\eta_p^2}\E_\mu\left[(\mL G_t)^2\right] + a(t)^2 \leq  \frac{4M^2}{\eta_p^2}\E_\mu\big[G_t^2\big] + a(t)^2.\]
Since $c^2= 1 - 4M^2/\eta_p^2>0$, we deduce $\E_\mu[G_t^2]\leq a(t)^2/c^2$. Finally, we have
\[\left(\E_\mu\big[(\mL G_t)^2\big]\right)^{1/2}\leq 2M\left(\E_\mu\big[G_t^2\big]\right)^{1/2}\leq \frac{2M}{c}\cdot a(t).\]
It then follows from \eqref{eqt:continuous_step_1} and \eqref{eqt:continuous_step_2} that
\[a'(t)\leq \frac{2\sigma M}{c\eta_p}\cdot a(t),\]
and therefore
\[a(t)\leq a(0) \exp\left(\frac{2\sigma M\,t}{c\,\eta_p}\right) = \exp\left(\frac{2\sigma M\,t}{c\,\eta_p}\right),\quad t\geq 0.\]
This completes the proof.
\end{proof}

We can then bound the mgf of $\int_0^tf(Z_s)\idiff s$ for sufficiently small $\theta$ again by replacing $f$ with $\theta f$ in Proposition \ref{prop:bound_of_a_t}, and we omit the obvious proof. 

\begin{corollary}\label{cor:continuous_mgf_bound}
Let $f: \Omega\to \mathbb{R}$ be a function such that $\E_\mu [f] = 0$, $|f|\leq M$, and $\Var_\mu[f]\leq \sigma^2$. For any $\theta\in \R$ with $|\theta|<\eta_p/(2M)$, it holds that 
\begin{equation}
\E_{Z_0\sim\mu}\left[\econst^{\theta F\left(Z_{[0,t]}\right)}\right]\leq \exp\left(\frac{2\sigma M\theta^2\,t}{c(\theta)\,\eta_p}\right),
\end{equation}
where $c(\theta):=(1-4\theta^2M^2/\eta_p^2)^{1/2}$.
\end{corollary}

\subsection{Bernstein-type tail bound}
With a bound on the mgf of $\int_0^tf(Z_s)\idiff s$ in hand, we can establish a Bernstein-type inequality for continuous-time Markov processes in the next theorem.

\begin{theorem}\label{thm:Bernstein_continuous}
Let $(Z_t)_{t\geq 0}$ be a Markov process on $\Omega$ with an invariant distribution $\mu$. Suppose that the corresponding Markov semigroup $(\mtx{P}_t)_{t\geq 0}$ admits an IP gap $\eta_p>0$. Let $\nu$ be the initial distribution of $(Z_t)_{t\geq0}$, i.e. $Z_0\sim\nu$, such that $\nu$ is absolutely continuous with respect to $\mu$ and that $\nu/\mu$ has a finite $p$-moment for some $p\in[1,+\infty]$. Let $f\in \Omega\to \mathbb{R}$ be a function such that $\E_\mu [f] = 0$, $|f|\leq M$, and $\Var_\mu[f]\leq \sigma^2$ for some $M,\sigma\in (0,+\infty)$. Then, for all $t\geq0$ and all $\delta\geq 0$, 
\begin{equation}
\Prob{\left|\frac{1}{t}\int_0^t f(Z_s)\idiff s\right| \geq \delta } \leq 2\|\nu/\mu\|_{L_{p,\mu}}\exp\left(-\frac{t\,\eta_p\,\delta^2}{4qM\sqrt{4\sigma^2+\delta^2}}\right),
\end{equation}
where $q=p/(p-1)\in[1,+\infty]$. 
\end{theorem}

\begin{proof} For any $\theta>0$, we have
\begin{align*}
\Prob{\frac{1}{t}\int_0^tf(Z_s)\idiff s \geq \delta } &\leq \econst^{-\theta t\delta} \E_{Z_0\sim\nu}\left[\econst^{\theta F\left(Z_{[0,t]}\right)}\right]\\
&= \econst^{-\theta t\delta} \E_{Z_0\sim\mu}\left[(\nu(Z_0)/\mu(Z_0))\econst^{\theta F\left(Z_{[0,t]}\right)}\right]\\
&\leq \econst^{-\theta t\delta} \|\nu/\mu\|_{L_{p,\mu}}\left(\E_{Z_0\sim\mu}\left[\econst^{q\theta F\left(Z_{[0,t]}\right)}\right]\right)^{1/q}.
\end{align*}
Corollary \ref{cor:continuous_mgf_bound} implies 
\[\E_{Z_0\sim\mu}\left[\econst^{q\theta F\left(Z_{[0,t]}\right)}\right]\leq  \exp\left(\frac{2tM\sigma q^2\theta^2}{c(q\theta)\,\eta_p}\right),\]
given that $c(q\theta)^2=1-4q^2\theta^2M^2/\eta_p^2> 0$. In particular, by choosing 
\[\theta = \frac{\delta\eta_p}{2qM\sqrt{4\sigma^2+\delta^2}},\]
which gives
\[c(q\theta) = \frac{2\sigma}{\sqrt{4\sigma^2+\delta^2}},\]
we obtain
\begin{align*}
\Prob{\frac{1}{t}\int_0^tf(Z_s)\idiff s \geq \delta } &\leq  \|\nu/\mu\|_{L_{p,\mu}}\exp\left(- t\delta \theta + \frac{2tM \sigma q \theta^2}{c(q\theta)\,\eta_p}\right) \\
&=  \|\nu/\mu\|_{L_{p,\mu}}\exp\left(-\frac{t\,\eta_p\,\delta^2}{4qM\sqrt{4\sigma^2+\delta^2}}\right).
\end{align*}
The claimed result then follows by doubling the bound to cover the whole deviation probability. 
\end{proof}

\subsection*{Acknowledgements} The authors are supported by the National Key R\&D Program of China under the grant 2021YFA1001500. We are very grateful for the constructive comments on this paper from Joel Tropp, Sam Power, and anonymous referees.

\appendix 

\section{Comparison of different spectral gaps}\label{sec:simple_example}
This section compares the IP gap required for our Bernstein-type inequality with other spectral gaps appearing in the existing literature. To start, we say a Markov transition operator $\mtx{P}$ admits a nonzero \textit{absolute spectral gap} $\eta_a=1-\lambda_a$ if 
\begin{equation}
\lambda_a:= \sup _{h \in L_{2,\mu}^0,\ h \neq \mtx{0}} \frac{\|\mtx{P} h\|_{\mu}}{\|h\|_{\mu}}<1.
\end{equation}
To the best of our knowledge, most of the existing results about concentration inequalities for non-reversible Markov chains require a nonzero absolute spectral gap of the transition operator. This condition, however, may fail in some applications. As a generalization of the absolute spectral gap, Paulin \cite{paulin2015concentration} introduced the \textit{pseudo spectral gap} that is defined as
\begin{equation}
\eta_{ps} := \sup_{k\in \N_+}\frac{1- \lambda_{a,k}^2}{k},
\end{equation}
where 
\begin{equation}
\lambda_{a,k} := \sup _{h \in L_{2,\mu}^0,\ h \neq \mtx{0}}  \frac{\|\mtx{P}^k h\|_{\mu}}{\|h\|_{\mu}}.
\end{equation}
Note that by Jensen's inequality,
\begin{equation}
\|\mtx{P} h\|_{\mu}^2=\sum\limits_{u\in \Omega}\mu(u)\left(\sum\limits_{v\in\Omega} p(u,v)h(v)\right)^2 \leq \sum\limits_{u\in \Omega}\mu(u)\sum\limits_{v\in\Omega}p(u,v) h(v)^2 = \|h\|_{\mu}^2,
\end{equation}
and thus we have $\lambda_{a,k}\leq \lambda_a^k\leq 1$ for all $k\in \N_+$.

Another quantity appearing in the literature is the \textit{symmetric spectral gap} (also called the right spectral gap in \cite{jiang2018bernstein,fan2021hoeffding}), which is defined to be the spectral gap of the additive symmetrization of $\mtx{P}$. To be specific, let $\mtx{P}^*$ be the adjoint of $\mtx{P}$ with respect to the $L_{2,\mu}$ inner product, and let $\mtx{A}:=(\mtx{P}+\mtx{P}^*)/2$ be the additive symmetrization of $\mtx{P}$. We say $\mtx{P}$ admits a nonzero \textit{symmetric spectral gap} $\eta_s=1-\lambda_s$ if 
\begin{equation}
\lambda_s:=\sup _{h \in L_{2,\mu}^0,\ h \neq \mtx{0}} \frac{\langle h, \mtx{P}h \rangle_{\mu}}{ \|h\|_{\mu}^2}= \sup _{h \in L_{2,\mu}^0,\ h \neq \mtx{0}} \frac{\langle h, \mtx{A}h \rangle_{\mu}}{ \|h\|_{\mu}^2}<1.
\end{equation}
This quantity is employed to develop concentration inequality for general Markov chains in \cite{fan2021hoeffding,jiang2018bernstein}. However, their proof seems to have some technical mistake, which we will discuss in the next section. 

Here, we argue that for any transition operator $\mtx{P}$, the relations $\eta_p\geq \eta_s\geq \eta_a$ and $\eta_p\geq \eta_{ps}/2$ always hold. The proof of $\eta_p\geq \eta_{ps}/2$ is due to Chatterjee \cite{chatterjee2025spectral}, and we still provide the proof below for the reader's convenience. As a direct consequence, our Bernstein-type inequality also holds if $\mtx{P}$ admits a nonzero absolute spectral gap, a nonzero pseudo spectral gap, or a nonzero symmetric spectral gap.

\begin{lemma}\label{lem:gap_comparison}
It always holds that $\eta_p\geq \eta_s\geq \eta_a$ and $\eta_p\geq \eta_{ps}/2$.
\end{lemma}

\begin{proof}
For any $h\in L_{2,\mu}^0$, by the Cauchy--Schwarz inequality,
\begin{equation}
1 - \frac{\langle h, \mtx{P}h \rangle_{\mu}}{ \|h\|_{\mu}^2} =\frac{\langle h, (-\mtx{L})h \rangle_{\mu}}{\|h\|_{\mu}^2} \leq \frac{\|\mtx{L}h\|_{\mu}}{\|h\|_{\mu}}.
\end{equation}
Taking the infimum over all $h\in L_{2,\mu}^0$, we get $\eta_s\leq \eta_p$. Similarly, since
\begin{equation}
\frac{\langle h, \mtx{P}h \rangle_{\mu}}{ \|h\|_{\mu}^2}\leq \frac{\|\mtx{P}h\|_{\mu}}{\|h\|_{\mu}}
\end{equation}
holds for any $h\in L_{2,\mu}^0$, we deduce $\lambda_s\leq \lambda_a$, and therefore $\eta_s\geq \eta_a$.

To prove $\eta_p\geq \eta_{ps}/2$, we first note that for any $k\in \N_+$,
\[\|(\mtx{P}^{k-1}-\mtx{P}^k)h\|_\mu = \|\mtx{P}^{k-1}(\Id - \mtx{P})h\|_\mu\leq \|(\Id - \mtx{P})h\|_\mu = \|\mtx{L}h\|_\mu.\]
Then, for any $h\in L_{2,\mu}^0$ and any $k\in \N_+$, we can compute by triangle inequality and telescoping that 
\[1 - \frac{\|\mtx{P}^kh\|_\mu}{\|h\|_\mu}\leq \frac{\|(\Id-\mtx{P}^k)h\|_\mu}{\|h\|_\mu}\leq \frac{1}{{\|h\|_\mu}}\suml_{i=1}^k\|(\mtx{P}^{i-1}-\mtx{P}^i)h\|_\mu\leq k\frac{\|\mtx{L}h\|_\mu}{{\|h\|_\mu}}.\]
It follows by Taking the infimum over all $h\in L_{2,\mu}^0$ that $1-\lambda_{a,k} \leq k\eta_p$, and thus
\[\eta_{ps} = \sup_{k\in \N_+}\frac{1- \lambda_{a,k}^2}{k} \leq  \sup_{k\in \N_+}\frac{2(1-\lambda_{a,k})}{k}\leq 2\eta_p.\]
We have used above the fact that $\lambda_{a,k}\leq 1$ for all $k\in \N_+$.
\end{proof}

When the number of states $|\Omega|$ is finite, $\lambda_s$ is the second largest eigenvalue of the matrix $\mtx{A}$. Moreover, if the transition matrix $\mtx{P}$ is irreducible, then it is not hard to see that $\mtx{A}$ is also an irreducible transition matrix. Therefore, by the Perron--Frobenius theorem, we always have $\lambda_s<1$, and thus $\eta_p\geq\eta_s=1-\lambda_s>0$. However, irreducibility does not guarantee a uniform lower bound of the ratios $\eta_a/\eta_p$ and $\eta_{ps}/\eta_p$. Below are some simple examples demonstrating the superiority of the IP gap.

\begin{example}
Chatterjee provided in \cite{chatterjee2025spectral} an example demonstrating that the ratio $\eta_a/\eta_p$ can be arbitrarily small. Consider the one dimensional random walk on the periodic grid $\Z/n\Z$ where the probability of not moving and that of moving one step to the right are both $1/2$. It is then not hard to check that the IP gap $\eta_p$ of this Markov chain is of order $1/n$, while the symmetric spectral gap $\eta_s$ and the absolute spectral $\eta_a$ are both of order $1/n^2$.
\end{example}

\begin{example} 
The next example (modified from Chatterjee's random walk example) also shows the superiority of the IP gap over the other gaps that have been used to establish concentration inequalities for non-reversible Markov chains. Consider a parameterized family of transition matrices,
\begin{equation}
\mtx{P}_0=\begin{bmatrix}
        0 & 1 & 0 \\
        0 & 0 & 1 \\
        1 & 0 & 0 \\
\end{bmatrix}
\quad\text{and}\quad
\mtx{P}_\varepsilon = \varepsilon \Id + (1-\varepsilon)\mtx{P}_0,\quad \varepsilon \in[0,1].
\end{equation}
Suppose $\varepsilon\ll 1$. One can easily show by a straightforward calculation and the spectral perturbation theory that 
\[\eta_p(\mtx{P}_\varepsilon) = (1-\varepsilon)\sqrt{3},\quad \eta_a(\mtx{P}_\varepsilon) = O(\varepsilon),\quad \eta_{ps}(\mtx{P}_\varepsilon) = O(\varepsilon).\]
In particular, $\eta_p(\mtx{P}_0)=\sqrt{3}$ and $\eta_a(\mtx{P}_0) = \eta_{ps}(\mtx{P}_0)=0$.
\end{example}

\begin{example}
Here, we give another simple example where $\eta_p>0$ but $\eta_a=0$. Consider the transition matrix
\begin{equation}
\mtx{P}=\begin{bmatrix}
        1/2& 1/2 &0    &0  \\
        0  &   0 &1/2  &1/2\\
        1/2& 1/2 &0    &0  \\
        0  &   0 &1/2  &1/2\\
\end{bmatrix}.
\end{equation}
One can check that $\mtx{P}$ is irreducible and has a unique invariant distribution $\mu=(1/4,1/4,1/4,1/4)$. By definition, $\lambda_a$ is the second largest singular value of the matrix $\mtx{D}_\mu^{1/2}\mtx{P}\mtx{D}_\mu^{-1/2}$, where $\mtx{D}_\mu$ is the diagonal matrix with the elements of $\mu$ on the diagonal. Since $\mtx{D}_\mu$ is a multiple of the identity matrix, $\mtx{D}_\mu^{1/2}\mtx{P}\mtx{D}_\mu^{-1/2}=\mtx{P}$. A direct calculation shows that the largest two singular values of $\mtx{P}$ are both $1$, and therefore the absolute spectral gap of $\mtx{P}$ is 0. 
\end{example}

\section{Potential proof gap in two related works}\label{sec:gap}
In \cite{fan2021hoeffding,jiang2018bernstein}, the authors rely on the following lemma to establish concentration inequalities for non-reversible Markov chains with a nonzero symmetric spectral gap $\eta_s$ (i.e. the right spectral gap in their notation):

\begin{lemma}[Fan, Jiang and Sun(2021), Lemma 11]\label{lem:bound_operator_norm}
Let $\left\{Z_k\right\}_{k \geq 1}$ be a Markov chain with invariant measure $\mu$ and right spectral gap $1-\lambda_s>0$. Then for any bounded function $f$ and any $t \in \mathbb{R}$,
\begin{equation}
\mathbb{E}_\mu\left[\econst^{t \sum_{k=1}^n f\left(Z_k\right)}\right] \leq\left\| E^{t f / 2} \hat{\mtx{P}}_{\max \left\{\lambda_s, 0\right\}} E^{t f / 2}\right\|_{\mu}^n.
\end{equation} 
\end{lemma}

Here $\hat{\mtx{P}}_{\lambda}=\lambda \mtx{\Id}+(1-\lambda)\mtx{\Pi}$, where $\mtx{\Pi}$ is the projection onto the 1-dimensional subspace spanned by $\mtx{1}$, and $E^f$ is the multiplication operator of function $\econst^f$, i.e., $E^f: h \in \mathcal{L}_2(\mu) \mapsto=\econst^f h$.
A key point of their proof is the application of the power inequality for the numerical radius of operators:
\begin{equation}\label{eqt:fan2021hoeffding_keystep}
\begin{split}
\mathbb{E}_\mu\left[\econst^{t \sum_{k=1}^n f\left(Z_k\right)}\right] & =\left\langle \econst^{t f / 2},\left(E^{t f / 2} \mtx{P} E^{t f / 2}\right)^{n-1} \econst^{t f / 2}\right\rangle_{\mu} \\
&\leq \sup _{\|h\|_\mu=1}\left|\left\langle\left(E^{t f / 2} \mtx{P}  E^{t f / 2}\right)^{n-1} h, h\right\rangle_{\mu}\right|\left\|e^{t f / 2}\right\|_{\mu}^2\\
&\leq\left(\sup _{\|h\|_\mu=1}\left|\left\langle E^{t f / 2} \mtx{P}  E^{t f / 2} h, h\right\rangle_{\mu}\right|\right)^{n-1}\left\|\econst^{t f / 2}\right\|_{\mu}^2, 
\end{split}
\end{equation}
where the last inequality is claimed to follow from a result of Pearcy in \cite{pearcy1966elementary}.

\begin{theorem}[Power inequality, Pearcy (1966)]\label{thm:power_inequality}
Let $\mathcal{H}$ be a complex Hilbert space, and let $\mtx{A}$ be a bounded operator on $\mathcal{H}$, then for any integer $n$, 
\begin{equation}\label{power_inequality}
w(\mtx{A}^n)\leq \left[w(\mtx{A})\right]^n,
\end{equation}
where $ w(\mtx{A})$ is the numerical radius of $\mtx{A}$: $w(\mtx{A}):=\sup_{\|x\|=1}\left|\langle\mtx{A}x,x\rangle\right|$.
\end{theorem}

The proof of Lemma \ref{lem:bound_operator_norm} in \cite{fan2021hoeffding} relies on the fact that the $L_{2,\mu}$ space considered there is a real Hilbert space. However, the original proof of Theorem \ref{thm:power_inequality} in \cite{pearcy1966elementary} only works for complex Hilbert spaces, and the power inequality \eqref{power_inequality} in real Hilbert spaces may not be true in general. For instance, let $\mathcal{H}=\R^n$, and let $\mtx{A}$ be a skew-symmetric matrix, then $w(\mtx{A})=0$ but $w(\mtx{A}^n)$ could be positive when $n\geq 2$. Therefore, if the supremums in \eqref{eqt:fan2021hoeffding_keystep} are taken over the real $\mL_{2,\mu}$ space, then the following inequality (that leads to the last inequality in \eqref{eqt:fan2021hoeffding_keystep}) cannot be derived from Theorem \ref{thm:power_inequality} directly:
\begin{equation}
\sup _{\|h\|_{\mu}=1}\left|\left\langle\left(E^{t f / 2} \mtx{P} E^{t f / 2}\right)^{n-1} h, h\right\rangle_{\mu}\right|\leq\left(\sup _{\|h\|_{\mu}=1}\left|\left\langle E^{t f / 2} \mtx{P} E^{t f / 2} h, h\right\rangle_{\mu}\right|\right)^{n-1}.
\end{equation}
On the other hand, if the supremums are taken over the complex $\mL_{2,\mu}$ space, then it should be explained in \cite{fan2021hoeffding} why
\begin{equation}
\sup _{\substack{\|h\|_{\mu}=1\\h \text{ is complex}}}\left|\left\langle E^{t f / 2} \mtx{P}  E^{t f / 2} h, h\right\rangle_{\mu}\right|=\sup _{\substack{\|h\|_{\mu}=1\\h \text{ is real}}}\left|\left\langle E^{t f / 2} \mtx{P}  E^{t f / 2} h, h\right\rangle_{\mu}\right|
\end{equation}
given that $E^{t f / 2} \mtx{P}  E^{t f / 2}$ is not symmetric, because their follow-up argument requires $h$ to be a real-valued function (vector). However, the above issue is not addressed in \cite{fan2021hoeffding}, and hence there might be a gap in their proof.

\vspace{2mm}

Next, we present the original proof of Theorem \ref{thm:power_inequality} in \cite{pearcy1966elementary} for the reader's convenience. 

\begin{proof}[Proof of Theorem \ref{thm:power_inequality}]
First note the two polynomial identities
\begin{equation}
1-z^n=\prod\limits_{k=1}^n(1-\omega_kz)\quad \text{and}\quad
1=\frac{1}{n}\sum\limits_{j=1}^n\prod \limits_{\substack{k=1\\k\ne j}}^n(1-\omega_kz),
\end{equation}
where $\omega_k=\econst^{2\pi k\iunit/n}$. For any unit vector $x\in\mathcal{H}$, let
\[x_j=\frac{1}{n}\prod \limits_{\substack{k=1\\k\ne j}}^n(\mtx{\Id}-\omega_k\mtx{A})x.\] 
Then we have $x=\sum_{j=1}^n x_j$. We can compute that
\begin{equation}\label{eqt:key_lemma_step1}
\begin{split}
1-\left\langle\mtx{A}^n x,x\right\rangle &=\left\langle\left(\mtx{\Id}-\mtx{A}^n\right)x,x\right\rangle=\sum\limits_{j=1}^n\left\langle\left(\mtx{\Id}-\mtx{A}^n\right)x,x_j\right\rangle\\
&=\sum\limits_{j=1}^n\left\langle\left(\mtx{\Id}-\omega_j\mtx{A}\right)x_j,x_j\right\rangle=\sum\limits_{j=1}^n\|x_j\|^2\left(1-\omega_j\left\langle\frac{\mtx{A}x_j}{\|x_j\|},\frac{x_j}{\|x_j\|}\right\rangle\right).
\end{split}
\end{equation}
To prove \eqref{power_inequality}, it suffices to assume that $w(\mtx{A})\leq 1$ and to prove $w(\mtx{A}^n)\leq 1$. Replacing $\mtx{A}$ in \eqref{eqt:key_lemma_step1} by $\econst^{i\theta}\mtx{A}$, one obtains the following relation
\begin{equation}\label{eqt:key_lemma_step2}
1-\econst^{in\theta}\left\langle\mtx{A}^n x,x\right\rangle=\sum\limits_{j=1}^n\|x_j\|^2\left(1-\omega_j\econst^{i\theta}\left\langle\frac{\mtx{A}x_j}{\|x_j\|},\frac{x_j}{\|x_j\|}\right\rangle\right),\quad \text{for any}\,\ \theta\in \R.
\end{equation}
Since $\omega(\mtx{A})\leq 1$, the real part of each term on the right-hand side must be nonnegative, which implies $\operatorname{Re}(1-\econst^{in\theta}\langle\mtx{A}^n x,x\rangle)\geq 0$ for any $\theta\in\R$. The conclusion follows immediately.
\end{proof}

\begin{remark}
To prove a real version of Theorem \ref{thm:power_inequality}, one must only assume and use that $\left|\langle\mtx{A}x,x\rangle\right|\leq 1$ for all real $x$ with $\|x\|=1$. However, for $n\geq 3$, the polynomial $1-z^n$ always has complex roots, and thus the $x_j$ defined in the proof above cannot always be real in general. Therefore, one is not able to control $\langle \mtx{A}x_j,x_j\rangle$ on the right-hand side in \eqref{eqt:key_lemma_step2}.
\end{remark}

\bibliographystyle{myalpha}
\bibliography{reference}

\end{document}